\documentclass[12pt]{amsart}
\usepackage{amsmath,amssymb}
\usepackage{amsfonts}
\usepackage{amsthm}
\usepackage{latexsym}
\usepackage{graphicx}

\def\Sing{\mbox{Sing}}

\def\p{\partial}

\def\R{\mathbb{R}}

\def\vv<#1>{\langle#1\rangle}

\def\XXint#1#2{\setbox0=\hbox{$#1{#2}{\int}$}{#2}\kern-.5\wd0 }

\def\XXint#1#2#3{{\setbox0=\hbox{$#1{#2#3}{\int}$}
     \vcenter{\hbox{$#2#3$}}\kern-.5\wd0}}

\def\Lip{\mbox{Lip}}

\def\vv<#1>{{\left\langle#1\right\rangle}}

\def\sph{\mathbb{S}}

\newtheorem{thm}{Theorem}[section]

\newtheorem{lem}{Lemma}[section]

\newtheorem{cor}{Corollary}[section]
\theoremstyle{definition}
\newtheorem{defn}{Definition}[section]
\theoremstyle{remark}

\numberwithin{equation}{section}

\begin{document}
\title{Critical points and surjectivity of smooth maps }

\author{Yongjie Shi$^1$}
\address{Department of Mathematics, Shantou University, Shantou, Guangdong, 515063, China}
\email{yjshi@stu.edu.cn}
\author{Chengjie Yu$^2$}
\address{Department of Mathematics, Shantou University, Shantou, Guangdong, 515063, China}
\email{cjyu@stu.edu.cn}
\thanks{$^1$Research partially supported by NSF of China with contract no. 11701355. }
\thanks{$^2$Research partially supported by NSF of China with contract no. 11571215.}
\renewcommand{\subjclassname}{%
  \textup{2010} Mathematics Subject Classification}
\date{}
\begin{abstract}
Let $f:M^m\to N^n$ be a smooth map between two differential manifolds with $N$ connected, $f(M)$ closed and $f(M)\neq N$. In this short note, we show that either all the points of $M$ are critical points of $f$ or the dimension the collection of all critical points of $f$ is not less than $n-1$. Some consequences of this result for surjectivity of mappings are also presented.
\end{abstract}
\maketitle\markboth{Shi \& Yu}{critical points and surjectivity}
\section{Introduction}
In \cite{LL}, the authors obtained the following interesting result:
\begin{thm}\label{thm-LL}
Let $M^m$ be a smooth manifold and  $f:M\to \R^n$ be a $C^1$-map with $n\geq 2$. If $f$ has finitely many critical points and $f(M)$ is a closed subset of $\R^n$, then $f$ is surjective. In particular, if $M$ is compact, then $f$ has infinitely many critical points.
\end{thm}

In this paper, by applying a similar trick as in the proof of Theorem \ref{thm-LL} in \cite{LL} and the observation that boundary points of $f(M)$ must be critical values, we are able to obtain a stronger conclusion:
\begin{thm}\label{thm-main}
Let $M$ be an $m$-dimensional differential manifold, $N$ be a connected $n$-dimensional differential manifold, and  $f:M\to N$ be a $C^1$-map  with  $f(M)$ closed and $f(M)\neq N$. Then either all points of $M$ are critical points of $f$ or the dimension of the collection of all critical points of $f$ is not less than $n-1$.
\end{thm}
As a consequence, we have the following criterion for surjectivity of differentiable maps:
\begin{cor}\label{cor-surj-sm}
Let $M$ be an $m$-dimensional differential manifold, $N$ be a connected $n$-dimensional differential manifold, and  $f:M\to N$ be a $C^1$-map  such that
\begin{enumerate}
\item $f$ has at least one regular point;
\item $f(M)$ is closed in $N$;
\item the dimension of the collection of all critical points is less than $n-1$.
\end{enumerate}
Then, $f$ is surjective.
\end{cor}
Note that assumption (1) in Corollary \ref{cor-surj-sm} is necessary for $f$ to be surjective by Sard's theorem. A direct consequence of Corollary \ref{cor-surj-sm} is the following slightly stronger version of Theorem \ref{thm-LL}. It seems that the original argument in \cite{LL} does not work for the proof of this stronger version since an infinite countable set of points may generate the whole space by taking closure.

\begin{cor}\label{cor-surj-sm-1}
Let $M$ be an $m$-dimensional differential manifold, $N$ be a connected $n$-dimensional differential manifold with $n\geq 2$, and  $f:M\to N$ be a $C^1$-map  such that
\begin{enumerate}
\item $f(M)$ is closed in $N$;
\item $f$ has only countably many critical points.
\end{enumerate}
Then, $f$ is surjective.
\end{cor}
Another direct corollary is the following  stronger version of Theorem \ref{thm-LL} for compact source and noncompact  target.
\begin{cor}
Let $M$ be a compact manifold, $N^n$ be an $n$-dimensional connected noncompact manifold and $f:M\to N$ be a $C^1$-map. Then either all points of $M$ are critical points of $f$ or the dimension of the collection of all critical points of $f$ is not less than $n-1$.
\end{cor}

Note that the collection of all  critical points of a holomorphic map $f$ between complex manifolds is a proper analytic subvariety and  must be at least of real codimension 2 when $f$ has regular points. We have the following criterion for surjectivity of holomorphic maps.
\begin{cor}\label{cor-surj-holo}
Let $M$ and $N$ be two connected complex manifolds of complex dimension $n$, and $f:M\to N$ be a holomorphic map such that
\begin{enumerate}
\item $f$ has at least a regular point;
\item $f(M)$ is closed in $N$.
\end{enumerate}
Then, $f$ is surjective.
\end{cor}
According to Sard's theorem, the converse of Corollary \ref{cor-surj-holo} is also true. In fact, the last conclusion can be extended to complex analytic spaces by using a similar argument as in the proof of Theorem \ref{thm-main}.
\begin{thm}\label{thm-surj-holo-1}
Let $X$ and $Y$ be two irreducible $n$-dimensional complex analytic spaces, and $f:X\to Y$ be a holomorphic map such that
\begin{enumerate}
\item $f$ has at least a regular point. More precisely, there is a smooth point $p$ of $X$ such that $f(p)$ is a smooth point of $Y$ and $f$ is nondegenerate at $p$;
\item $f(X)$ is closed.
\end{enumerate}
Then, $f$ is surjective.
\end{thm}
In the last result, when assumption (2) is replaced by a stronger assumption that $f$ is proper, the conclusion can be drawn directly from Remmert's proper mapping theorem (see \cite{D}).

The proofs of the results above are quite elementary. The key step is the following simple lemma.
\begin{lem}\label{lem-key}
Let $M^n$ be a connected differential manifold and $E$ be a proper closed subset of $M$ with nonempty interior. Then the dimension of $\p E$ is not less than $n-1$.
\end{lem}

The rest of this paper is organized as follows: in Section 2, we give some preliminaries about an intrinsic theory for dimension of subsets of differential manifolds and complex analytic space; In Section 3, we prove Lemma \ref{lem-key} , Theorem \ref{thm-main} and Theorem \ref{thm-surj-holo-1}.
\section{dimension of subsets}
By Whitney's embedding theorem, we know that every differential manifold $M$ can be embedded into some Euclidean space $\R^N$. So, for a subset $E$ of $M$, we can define the dimension of $E$ as the Hausdorff dimension of $E$ in $\R^N$. It is not hard to check that this definition of the dimension of $E$ is independent of the embedding and so is well defined. This is the ambient view of dimension of subsets in a differential manifold.

On the other hand, by combining the idea of defining subsets of measure zero in a differential manifold without given a specific measure (see \cite{GP,Z}) and the definition of Hausdorff dimension of a subset in Euclidean spaces (see \cite{E}), we can define the dimension of a subset of a differential manifold intrinsically in the following way. We believe that this is well known for experts. However, we can not find references for this intrinsic definition of dimensions for subsets of differential manifolds. So, we give a short description on this.

\begin{defn}
Let $M^n$ be a differential  manifold and $\mathcal H^s$ be the $s$-dimensional Hausdorff measure on $\R^n$. A subset $E\subset M$ is said to be of $\mathcal H^s$-measure zero if for any coordinate chart $(U,\varphi)$ of $M$, $$\mathcal H^s(\varphi(E\cap U))=0.$$
\end{defn}

By the fact (see \cite{E}) that
\begin{equation}
\mathcal H^s(f(A))\leq (\Lip(f))^s\mathcal H^s(A)
\end{equation}
for any $A\subset \R^m$ and Lipschitz map $f:\R^m\to \R^n$, where
\begin{equation}
\Lip(f)=\sup_{x\neq y}\frac{\|f(x)-f(y)\|}{\|x-y\|},
\end{equation}
it is not hard to show the following:
\begin{thm}\label{thm-measure-zero}
Let $M$ and $N$ be two differential manifolds, $f:M\to N$ be a $C^1$-map and $A\subset M$ be of $\mathcal H^s$-measure zero. Then, $f(A)$ is also of $\mathcal H^s$-measure zero.
\end{thm}
Note that the Hausdorff dimension of a subset $A$ in an Euclidean space is defined as
\begin{equation}
\inf\{s>0\ |\ \mathcal H^s(A)=0.\}.
\end{equation}
Imitating this, we define the dimension of a subset of a differential manifold as follows.
\begin{defn}
Let $M$ be a differential manifold and $E\subset M$. Define the dimension of $E$ as
\begin{equation}
\dim E=\inf\{s>0\ |\ E\ \mbox{is of $\mathcal H^s$-measure zero.}\}.
\end{equation}
\end{defn}
By Theorem \ref{thm-measure-zero}, one has the following straight forward conclusion.
\begin{thm}\label{thm-dim}
Let $M$ and $N$ be two differential manifolds, $f:M\to N$ be a $C^1$-map and $E\subset M$. Then, $\dim f(E)\leq \dim E$.
\end{thm}

Note that the definition of complex analytic spaces (see \cite{W}) has a similar feature with the definition of differential manifolds using local charts. So, one can define the dimension of a general subset of a complex analytic space similarly , and Theorem \ref{thm-measure-zero} and \ref{thm-dim} remain true in this category.
\section{Proof of results}
We first prove Lemma \ref{lem-key}
\begin{proof}[Proof of Lemma \ref{lem-key}]
Let $p$ be an interior point of $E$ and $q\in E^c$. Let $(U,\varphi)$ be a coordinate chart at $q$ such that $\varphi(q)=o$ and $\varphi(U)=B_o(2)$ where $o$ is the origin of $\R^n$ and $B_o(2)$ is the ball of radius $2$ centered at $o$. Let $p'\in U$ be such that $\varphi(p')=e_1\in \R^n$. Here $e_1=(1,0,\cdots,0)$. By homogeneity of  differential manifolds (see \cite{Z}), there is a diffeomorphism $\psi:M\to M$ such that $\psi(p)=p'$ and $\psi (q)=q$.

Let $\pi:\R^{n}\setminus\{o\}\to \sph^{n-1}$ be such that $\pi(x)=\frac{x}{\|x\|}$. It is clear that $e_1$ is an interior point of $\varphi(\psi(E)\cap U)$. Suppose that $B_{e_1}(\delta)\subset \varphi(\psi(E)\cap U)$ for some $\delta>0$. Then, on each line segment joining $x\in B_{e_1}(\delta)$ and $o$, there is at least one point in $\varphi(\psi(\p E)\cap U)$. This implies that $\pi(\varphi(\psi(\p E)\cap U))$ contains a neighborhood of $e_1$ in $\sph^{n-1}$. So, $$\dim \varphi(\psi(\p E)\cap U)\geq n-1 $$ by Theorem \ref{thm-dim}. This completes the proof of the lemma.
\end{proof}

We are now ready to prove Theorem \ref{thm-main}.
\begin{proof}[Proof of Theorem \ref{thm-main}] If $f$ has a regular point $p$, then $f(p)$ is an interior point of $f(M)$ by the local submersion theorem (see \cite{GP,Z}). By Lemma \ref{lem-key}, $\dim \p f(M)\geq n-1$. By Theorem \ref{thm-dim}, $\dim f^{-1}(\p f(M))\geq n-1$. Note that every point in $f^{-1}(\p f(M))$ is a critical point of $f$ by the local submersion theorem again. We complete the proof of the theorem.
\end{proof}
Finally, we come to prove Theorem \ref{thm-surj-holo-1}.
\begin{proof}
Suppose that $f(X)\neq Y$. Because $Y$ is irreducible, the smooth part $Y^*$ of $Y$ is connected (see \cite{GH}). By that $f$ has a regular point and $f(X)$ is closed in $Y$, $f(X)\cap Y^*$ is a proper closed subset of $Y^*$ with nonempty interior. So, by Lemma \ref{lem-key}, $\dim \p(f(X)\cap Y^*)\geq 2n-1$. Then,
\begin{equation}
\dim f^{-1}(\p(f(X)\cap Y^*))\geq 2n-1.
\end{equation}
Note that the singular part $\Sing X$ of $X$ as a proper subvariety must have
\begin{equation}
\dim \Sing X\leq 2n-2.
\end{equation}
Therefore,
\begin{equation}
\dim f^{-1}(\p(f(X)\cap Y^*))\cap X^*=\dim f^{-1}(\p(f(X)\cap Y^*))\setminus\Sing X\geq 2n-1.
\end{equation}
Note that any point in $f^{-1}(\p(f(X)\cap Y^*))\cap X^*$ is a critical point of $f:X^*\cap f^{-1}(Y^*)\to Y^*$ by the local submersion theorem. However, the collection of all critical points of $f:X^*\cap f^{-1}(Y^*)\to Y^*$ as a proper subvariety of $X^*\cap f^{-1}(Y^*)$ must be of dimension no more than $2n-2$. This is a contradiction. So $f$ must be surjective.
\end{proof}

\end{document}